\documentclass[11pt]{article}
\usepackage[margin=2.4cm]{geometry}

\usepackage{graphicx}
\graphicspath{{./images}}
\usepackage{amsmath,amsthm,latexsym,amsfonts,amssymb,mathrsfs}
\usepackage[usenames]{color}
\usepackage{tikz}
\usetikzlibrary{math}

\usepackage{cancel,soul,ulem}
\usepackage{float}
\usepackage{verbatim}
\usepackage{booktabs}

\def\dsp{\displaystyle}

\numberwithin{equation}{section}

\newtheorem{theorem}{Theorem}[section]
\newtheorem{lemma}[theorem]{Lemma}

\newtheorem{remark}{Remark}[section]

\title{Numerical approximations for a hyperbolic integrodifferential equation with a non-positive variable-sign kernel and nonlinear-nonlocal damping}
\author{ 
Wenlin Qiu\thanks{School of Mathematics, Shandong University, Jinan 250100, China. (Email: wlqiu@sdu.edu.cn)}
\and Xiangcheng Zheng\thanks{School of Mathematics, Shandong University, Jinan 250100, China. (Email: xzheng@sdu.edu.cn)}
\and Kassem Mustapha\thanks{Corresponding author. School of Mathematics and Statistics, University of New South Wales, Sydney, Australia. (Email: 
k.mustapha@unsw.edu.au)}}

\begin{document}
\maketitle
\begin{abstract}
This work considers the Galerkin approximation and analysis for a hyperbolic integrodifferential equation, where the non-positive variable-sign kernel and nonlinear-nonlocal damping with both the weak and viscous damping effects are involved. We derive the long-time stability of the solution and its finite-time uniqueness. For the semi-discrete-in-space Galerkin scheme, we derive the long-time stability of the semi-discrete numerical solution and its finite-time error estimate by technical splitting of intricate terms. Then we further apply the centering difference method and the interpolating quadrature to construct a fully discrete Galerkin scheme and prove the long-time stability of the numerical solution and its finite-time error estimate by designing a new semi-norm. Numerical experiments are performed to verify the theoretical findings. 
\end{abstract}


\section{Introduction}
This work considers the numerical approximation and analysis of a hyperbolic integrodifferential equation \cite{Cannarsa1,Cannarsa} which models viscoelastic systems with memory in various fields, see e.g. \cite{Alabau-Boussouira,Appleby,Cannarsa1,Georgiev,Lebon,Renardy}
\begin{equation}\label{eq1.1}
   \begin{split}
       u''(x, t) + q(t)u'(x, t)  -\Delta u(x,t) + \int_0^{t}\beta(t-s)\Delta u(x,s)ds = f(x, t),
   \end{split}
   \end{equation}
for  $ (x,t)\in\Omega \times (0,T]$, subject to homogeneous Dirichlet boundary condition $u =0$ on  $\partial\Omega \times (0,T]$. Here,  $u'=\frac{\partial u}{\partial t}$, $u''=\frac{\partial^2 u}{\partial t^2}$, $T$ may be finite or infinite that will be specified in different cases, $\Omega$ is a bounded domain in $\mathbb{R}^d$ ($d=1,2,3$) with smooth boundary $\partial\Omega$, $\Delta$ is the Laplace operator,  and $f$ is the forcing function.  This model involves the nonlinear-nonlocal damping term with weak and viscous damping coefficients $\mu_1$ and $\mu_2$, respectively (cf. \cite[Equation 6.2]{Cannarsa} and \cite{Emm})
\begin{equation}\label{eq1.2}
     \begin{split}
          q(t) = G\left(   \int_{\Omega}\mu_1| u(x,t)|^2+\mu_2|\nabla u(x,t)|^2dx  \right), \quad \mu_1, \mu_2 \geq 0, \quad \mu_1^2+ \mu_2^2 > 0,
     \end{split}
     \end{equation}
and initial  conditions
     \begin{equation}\label{eq1.3}
          u(x,0)=u_0(x),  \quad  u'(x,0)=u_1(x),  \quad  x\in \Omega,
     \end{equation}
where $\nabla$ is the gradient operator,  the function $G:\mathbb{R}^+\rightarrow \mathbb{R}^+$, and $u_0$ and $u_1$ are given initial values. Moreover, the variable-sign convolution kernel $\beta\in L^1(0,\infty)$ is selected as \cite{Cannarsa1,Cannarsa}:
     \begin{equation}\label{eq1.5}
     \begin{split}
        &  \beta(t) = \frac{e^{-\sigma t}t^{\alpha-1}\cos(\gamma t)}{\Gamma(\alpha)},   \quad   \text{satisfying the following standard conditions} \\
        & \text{(i)} \; \alpha = 1, \quad \sigma> 1, \quad 0\leq \gamma \leq \sigma, \\
        & \text{(ii)} \; \alpha = 1/2, \quad \sigma> 1, \quad 0\leq \gamma \leq \sqrt{3}\sigma.
     \end{split}
     \end{equation}

See \cite{Cannarsa1,Cannarsa,Messaoudi1,Messaoudi2} for the well-posedness  of certain nonlinear hyperbolic integrodifferential equations. For the numerical aspect, Yanik and Fairweather \cite{Yanik} developed a Galerkin finite element method for a variant of (\ref{eq1.1}) where the nonlinearity appears in the convolution term. Cannon et al. \cite{Cannon} established the stability and $L^2$-error estimate of the Galerkin approximation for the linear version of (\ref{eq1.1}) with smooth kernel. Allegretto and Lin \cite{Allegretto} extended the results to the linear case with the positive-type or monotonic kernel. Fairweather \cite{Fairweather} formulated the orthogonal spline collocation methods and derived error estimates. Larsson and Saedpanah developed the continuous Galerkin method for the linear case of (\ref{eq1.1}) with weakly singular kernels \cite{Larsson,Saedpanah}.  Karaa et al. \cite{Karaa} proved a priori hp-estimate under the case of smooth kernels. Xu \cite{Xu2} deduced the long-time decay properties of the numerical solutions with tempered weakly singular kernels. Then, Xu \cite{Xu4} proved the long-time $L^{\infty}$ error estimate for a homogeneous viscoelastic rod. Recently, Baker et al. considered the convolution quadrature method for the wave equation with linear or nonlinear time-fractional damping terms \cite{Baker1,Baker2}.

Despite the significant progress on the linear case of (\ref{eq1.1}), there are rare studies considering the hyperbolic integrodifferential equation with variable-sign kernel and nonlinear-nonlocal damping coefficient. Due to the variability of the sign and the non-positivity of the kernel, many commonly-used quadrature rules for convolution with a completely monotonic kernel such as the rectangle rule are not applicable. Furthermore, the nonlinear-nonlocal damping makes the numerical analysis more intricate, especially in the case of long-time estimates that have significant applications such as the long-time performance tests of viscoelastic materials. 

This work accommodates these issues to propose and analyze a formally second-order accurate numerical approximation for problem \eqref{eq1.1}-\eqref{eq1.5}. Specifically, the main contributions are enumerated as follows:
\begin{itemize}
\item  We introduce a transformed kernel to convert the original model with a non-positive kernel into another hyperbolic integrodifferential equation with a positive type kernel, and overcome the difficulties caused by the variable-sign kernel and the nonlinear-nonlocal damping to prove the long-time stability and finite-time uniqueness of the solutions of problem \eqref{eq1.1}-\eqref{eq1.5}, the methods of which could assist the subsequent analysis of numerical solutions. 

\item We construct a spatial semi-discrete Galerkin scheme for solving \eqref{eq1.1}-\eqref{eq1.5} and derive the long-time stability of numerical solutions and the finite-time uniqueness and error estimate. A key ingredient of the proof lies in estimating the difference between nonlinear-nonlocal terms in (\ref{eq5.2}). Instead of splitting this difference into several pieces by introducing intermediate terms, which is difficult due to the nonlinearity and nonlocality and may deteriorate the accuracy, we introduce a quantity $\varphi_\vartheta$ to rewrite this difference into a more integrating form. By this means, a novel splitting of the difference (cf. (\ref{eq5.3})--(\ref{eq5.4})) appears to facilitate the analysis, cf. the estimates among (\ref{eq5.5})--(\ref{eq5.14}).

\item We employ the centering difference method and the interpolating quadrature to construct a fully discrete Galerkin scheme. We deduce the long-time stability of numerical solutions and the finite-time uniqueness and error estimate. Apart from the techniques in studying the semi-discrete scheme, we define a new semi-norm by (\ref{eq6.3}) to clarify the structure of the numerical scheme and the error equation to support the error estimate.
\end{itemize}

The rest of the work is organized as follows:
Section \ref{sec2} addresses some well-posedness issues of the model. Section \ref{sec3} gives long-time stability estimates and finite-time error estimates for the spatial semi-discrete scheme. The fully-discrete scheme is constructed and analyzed in Section \ref{sec4}. Section \ref{sec5} provides numerical experiments to verify the theoretical results, and we address concluding remarks in the last section.

\section{Well-posedness issue} \label{sec2}
Let $L^2(\Omega)$, $H^j(\Omega)$ ($j \ge 1$) and $H^1_0(\Omega)$  denote usual Sobolev spaces on $\Omega$. The notation $(\cdot, \cdot)$ indicates the $L^2(\Omega)$-inner product on $\Omega$ and  $\|\cdot\|$ is the associated norm, and $\|\cdot\|_j$ denotes the norm in $H^j(\Omega)$ for $j\ge 1.$ For $T>0$, the $L^{1}(0,T; L^2(\Omega))$ indicates the space of measurable
functions $g: [0,T] \rightarrow L^2(\Omega)$ satisfying $\int_{0}^{T}\|g(t)\|dt< \infty$. For the normed linear space $\mathcal{X}$ endowed with the norm $\|\cdot\|_{\mathcal{X}}$, the $ {\mathcal C}^{j}([0, T]; \mathcal{X})$ ($j=0,1$) represents the space of continuous or continuously differentiable functions from $[0, T]$ to $\mathcal{X}$, respectively. In particular,  we set $ {\mathcal C}([0,T]; \mathcal{X})= {\mathcal C}^{0}([0,T]; \mathcal{X})$. Throughout the work, we use $C$ to denote a generic positive constant independent of the mesh sizes and may differ at different occurrences.

To prove the main results, the following lemma for the variable-sign kernels in \eqref{eq1.5} is critical (the proof could be found in \cite{Cannarsa1} for the case (i) in \eqref{eq1.5} and in \cite{Zhao} for the case (ii) in \eqref{eq1.5}).
\begin{lemma}\label{lemma1.1} The kernel $K(t)=\int_t^{\infty}\beta(s)ds$ is of positive type with $K(\infty)=0$ and $K(0):=K_0<1$.
 \end{lemma}

 We impose  the following assumptions on the function  $G$ in \eqref{eq1.2}: for $z\ge 0,$

($\mathbf{A1}$)   $G(z)\geq g_0$  for some positive constant $g_0$, \par
($\mathbf{A2}$) $G$ is continuously differentiable with $0\leq G'(z)\leq \mathcal{L}$, where $\mathcal{L}$ is the Lipschitz constant. 

Examples of $G$  include $G(z)=1+z$ or $G(z)=\sqrt{1+z}$.  Note that the assumption $(\mathbf{A2})$ implies that for any $0\leq z\leq k$ with $k>0$, $G(z)=G(0)+G'(\xi)z\leq G(0)+\mathcal L k=:g_1(k)$. In subsequent analysis, we write $g_1(k)$ as $g_1$ for simplicity if $k$ depends on the solutions or given data.

For each $t \in [0,T],$ the weak solution $u$ 
of \eqref{eq1.1} is defined as:  
\begin{multline}\label{eq2.3}
          (u''(t),v)  + G\left(   \mu_1\| u(t)\|^2+\mu_2\|\nabla u(t)\|^2 \right) (u'(t),v) +  (\nabla u(t), \nabla v)  \\
         - \int_0^{t}\beta(t-s)(\nabla u(s),  \nabla v)ds = (f(t), v),
   \end{multline}
 for any $v \in H^1_0(\Omega)$ and any $t\in (0,T]$.
We then apply the energy method, see e.g. \cite[Corollary 4.7]{Cannarsa}, to perform mathematical analysis for problem \eqref{eq1.1}-\eqref{eq1.5}.

\subsection{Long-time stability} 

\begin{theorem}\label{theorem2.1}
    Under assumption $(\mathbf{A2})$, for $u_0\in H^1(\Omega)$, $u_1\in L^2(\Omega)$ and $f\in L^{1}(0,T;L^2(\Omega))$, the weak solution $u\in {\mathcal C}^1([0,T]; L^{2}(\Omega))\cap {\mathcal C}([0,T]; H^1_0(\Omega))$ is unique and  satisfies the  regularity estimate: 
\begin{equation}\label{eq2.5}
           \|u'(t)\|  + \|\nabla u(t)\| \leq C\left( \|u_1\|+ \|\nabla u_0\|  + 
        \int_0^{t} \|f(s)\|ds \right), \quad {\rm for}~~~ 0<t\leq T. 
      \end{equation}
      Furthermore, if $u_0\in H^2_0(\Omega)$, $u_1\in H^1(\Omega)$ and $f'\in L^{1}(0,T;L^2(\Omega))$, then $u(t) \in H^2(\Omega)$ for each $t\in (0,T],$ and  
\begin{equation}\label{eq2.5n}
    \|\nabla u'(t)\| \leq C\left( \|\nabla u_1\| + \|\Delta u_0\| + \|f(0)\| + \int_{0}^{t}(\|f(s)\|+\|f'(s)\|)ds  \right),\quad {\rm for}~~~ 0<t\leq T.
\end{equation}
\end{theorem}
\begin{proof}
In \eqref{eq2.3}, we choose $v=u'(t)$ to get
\begin{multline*}
          \frac{1}{2}\frac{d}{dt}\|u'(t)\|^2  + G\left(   \mu_1\| u(t)\|^2+\mu_2\|\nabla u(t)\|^2 \right) \|u'(t)\|^2 + \frac{1}{2}\frac{d}{dt}\|\nabla u(t)\|^2   \\
          - \int_0^{t}\beta(t-s)(\nabla u(s),  \nabla u'(t))ds = (f(t), u'(t)).
     \end{multline*}
Hence, multiply through by $2$, integrate over the time interval $(0,t)$ and use the assumption ($\mathbf{A1}$), $K'(t)=-\beta(t)$ and the Cauchy-Schwarz inequality to get
\begin{multline}
    \label{eq4.4}
          \|u'(t)\|^2  + 2g_0 \int_0^{t}\|u'(s)\|^2ds +\|\nabla u(t)\|^2  
        \\ + 2\int_0^{t}\int_0^{s}K'(s-z)(\nabla u(z),  \nabla u'(s))dzds 
         \le  \|u_1\|^2 + \|\nabla u_0\|^2  +   2\int_0^{t} \|f(s)\|  \| u'(s)\|ds.
\end{multline}
Integration by parts yields 
   \begin{multline*}
       \Phi(t;u) := \int_{0}^{t}\int_0^sK'(s-z)(\nabla u(z), \nabla u'(s))dzds \\
        = \left( \nabla u_{0}, \int_{0}^{t}K(s)\nabla u'(s)ds \right) - \frac{K_0}{2}\left[\|\nabla u(t)\|^2 - \|\nabla u_{0}\|^2\right]\\
        + \int_0^{t}\int_0^{s}K(s-z)(\nabla u'(q),  \nabla u'(s))dzds.
  \end{multline*}
Since the last term is nonnegative (by Lemma \ref{lemma1.1}) and since
\[\int_{0}^{t}K(s)\nabla u'(s)ds
       = K(t)\nabla u(t) - K_0\nabla u_0 + \int_{0}^{t}\beta(s)\nabla u(s)ds,\]
   \begin{equation}\label{eq4.8}
       \Phi(t; u) 
        \geq K(t)\left( \nabla u_{0}, \nabla u(t) \right) + \int_{0}^{t}\beta(s)(\nabla u_0,\nabla u(s)) ds  
         - \frac{K_0}{2}\left[\|\nabla u(t)\|^2 + \|\nabla u_{0}\|^2\right].
  \end{equation}
Inserting this in \eqref{eq4.4} and using again the Cauchy-Schwarz inequality,  we have
\begin{multline}\label{eq4.9}
          \|u'(t)\|^2  + 2g_0 \int_0^{t}\|u'(s)\|^2ds + \mu_0\|\nabla u(t)\|^2  
         \leq \|u_1\|^2 + (1+K_0)\|\nabla u_0\|^2  \\
        + 2\int_0^{t} \|f(s)\|  \| u'(s)\|ds 
         + 2K(t)\| \nabla u_{0}\| \|\nabla u(t) \| + 2\int_{0}^{t}\beta(s)\|\nabla u_0\|  \|\nabla u(s)\| ds.
     \end{multline}
Applying the Young's inequality and using  $K(t)\leq c_0$, then the last three terms on the right-hand side are bounded by 
\begin{multline*}
2\left(\int_0^{t} \|f(s)\|ds\right)^2 + \frac{1}{2}\|u'(\Bar{t})\|^2+
\frac{\mu_0}{2} \|\nabla u(t) \|^2 + \frac{2c_0^2}{\mu_0} \| \nabla u_{0}\|^2,    \\
+2\int_{0}^{t}\beta(s)ds \|\nabla u_0\|^2 + \frac{1}{2}\int_{0}^{t}\beta(s)  \|\nabla u(s)\|^2 ds, 
\end{multline*}
where $\|u'(\Bar{t})\|=\sup\limits_{0\leq s \leq t}\|u'(s)\|$ for some $0\le \Bar{t}\le t$. Therefore, with $c_1:=\max\limits_{0\leq t \leq \infty}\int_{0}^{t}|\beta(s)|ds $, we have
\begin{multline}\label{eq4.11}
          \|u'(t)\|^2   + \frac{\mu_0}{2}\|\nabla u(t)\|^2   
         \leq \|u_1\|^2+ \\
         \left(1+c_0+2c_1 + \frac{2c_0^2}{\mu_0} \right)\|\nabla u_0\|^2    
         +  \frac{1}{2}\int_{0}^{t}\beta(s)  \|\nabla u(s)\|^2 ds + 2\left(\int_0^{t} \|f(s)\|ds\right)^2 + \frac{1}{2}\|u'(\Bar{t})\|^2.
     \end{multline}
Choosing $t=\Bar{t}$ yields after some simple simplifications 
\[       \frac{1}{2}\|u'(\Bar{t})\|^2 
         \leq \|u_1\|^2 + \left(1+c_0+2c_1 + \frac{2c_0^2}{\mu_0} \right)\|\nabla u_0\|^2   
         +  \frac{1}{2}\int_{0}^{\Bar{t}}\beta(t)  \|\nabla u(t)\|^2 dt + 2\left(\int_0^{\Bar{t}} \|f(s)\|ds\right)^2.\]
Substituting this into the right-hand side of \eqref{eq4.11}, we obtain
\[\|u'(t)\|^2   + \frac{\mu_0}{2}\|\nabla u(t)\|^2   
         \leq C(\|u_1\|^2 +\|\nabla u_0\|^2)   
         + 4\left(\int_0^{t} \|f(s)\|ds\right)^2  +  \int_{0}^{t}\beta(s)  \|\nabla u(s)\|^2 ds.\]
Since $\int_{0}^{\infty}|\beta(t)| dt \leq c_1$, the application of Gr\"{o}nwall's lemma yields \eqref{eq2.5}. 

From  \cite[Proposition 4.4]{Cannarsa}, we conclude  that $u(t) \in H^2(\Omega)$ for each $t\in (0,T].$ To show the estimate in  \eqref{eq2.5n}, we choose $v=-\Delta u'(t)$ then follow the above derivation and  use the identity 
\[\int_0^t ( f(s), \Delta u'(s))ds=( f(t), \Delta u(t))-( f(0), \Delta u(0))-\int_0^t ( f'(s), \Delta u(s))\,ds,\]
in addition to the first achieved estimate. 
\end{proof}

\subsection{Finite-time uniqueness} 
Let $\hat{u},\tilde{u}$ be two solutions of \eqref{eq2.3}. We need to show that the difference $Y(t):=\hat{u}-\tilde{u}=0$.  From \eqref{eq2.3}, 
\begin{multline}
          (Y''(t),  v) +  G_0(\hat{u},\nabla \hat{u})(Y'(t), v) + (\nabla Y(t), \nabla v) \\
          - \int_0^{t}\beta(t-s)(\nabla Y(s),  \nabla v)ds = - [G_0(\hat{u},\nabla \hat{u})-G_0(\tilde{u},\nabla \tilde{u})]  (\tilde{u}'(t), v),\label{eq4.14}
     \end{multline}
where 
\begin{equation}\label{eq4.16}
     \begin{split}
           G_0(u(t),\nabla u(t)):=G\left(  \mu_1\| u(t)\|^2+\mu_2\|\nabla u(t)\|^2 \right)\geq g_0.
     \end{split}
     \end{equation}
Choose $v=Y'(t)$ in \eqref{eq4.14} to get
\begin{multline}\label{eq4.17}
          \frac{1}{2}\frac{d}{dt}\|Y'(t)\|^2  + g_0 \|Y'(t)\|^2 + \frac{1}{2}\frac{d}{dt}\|\nabla Y(t)\|^2   
           + \int_0^{t}K'(t-s)(\nabla Y(s),  \nabla Y'(t))ds \\
         \leq \left|G_0(\hat{u},\nabla \hat{u}) - G_0(\tilde{u},\nabla \tilde{u})\right| \|\tilde{u}'(t)\| \|Y'(t)\|.
     \end{multline}
By \eqref{eq2.5}, the assumption ($\mathbf{A2}$) and Poincar\'{e} inequality, we obtain
\begin{equation}\label{eq4.18}
     \begin{split}
         \left|G_0(\hat{u},\nabla \hat{u}) - G_0(\tilde{u},\nabla \tilde{u})\right| \|\tilde{u}'(t)\|
          \leq C( \| Y(t)\| + \|\nabla Y(t)\| ) \leq C\|\nabla Y(t)\|.
     \end{split}
     \end{equation}
 Thus, integrate \eqref{eq4.17} over the time interval $(0,t_{*})$, and then,  use $Y'(0)=Y(0)=0$, \eqref{eq4.8} and \eqref{eq4.18}, we  get
\begin{equation*}
     \begin{split}
          \|Y'(t_*)\|^2 & + 2 g_0 \int_0^{t_*}\|Y'(t)\|^2dt + \mu_0\|\nabla Y(t_*)\|^2   \leq C \int_0^{t_*}\|\nabla Y(t)\| \|Y'(t)\| dt \\
        &  
        \leq g_0 \int_0^{t_*}\| Y'(t)\|^2 dt + C\int_0^{t_*}\|\nabla Y(t)\|^2 dt.
     \end{split}
     \end{equation*}
Canceling the similar terms yields 
\[        \|Y'(t_*)\|^2  + g_0 \int_0^{t_*}\|Y'(t)\|^2dt + \mu_0\|\nabla Y(t_*)\|^2   \leq C\int_0^{t_*}\|\nabla Y(t)\|^2 dt.\]
We finally apply the Gr\"{o}nwall's lemma to complete the proof.

\section{Spatial semi-discrete scheme}\label{sec3} Given a quasiuniform partition of $\Omega$ with the diameter $h$ and let $\mathcal{S}_h\subset H^1_0(\Omega)$ be a finite-dimensional space with respect to this partition, with the following approximation property \cite{ McLean1, Mustapha, Thomee}
     \begin{equation*}
     \begin{array}{ll}
          \inf\limits_{\psi\in \mathcal{S}_h}\{ \|v-\psi\| + h\|\nabla(v-\psi)\|  \} \leq Ch^2 \|v\|_2, \quad v\in H^2(\Omega)\cap H^1_0(\Omega).
     \end{array}
     \end{equation*}
For  $t \in [0,T],$ the semi-discrete Galerkin finite element solution $u_h(t)\in \mathcal{S}_h$ is determined by
   \begin{align}\label{eq2.6}
       (u''_h(t),v_h)  & + G\left(   \mu_1\left\|u_h(t)\right\|^2+\mu_2\left\|\nabla u_h(t)\right\|^2 \right) (u'_h(t),v_h) +  (\nabla u_h(t),\nabla v_h) \\
        &\nonumber- \int_0^{t}\beta(t-s)(\nabla u_h(s),\nabla v_h)ds = (f(t), v_h), \quad {\rm for~any}~~v_h\in \mathcal{S}_h,
   \end{align}
with $u_h(0) = u_{0h}:= w_h(0)$ and $ u'_h(0) = u_{1h}:=w_h'(0)$, where 
for each $t\in (0,T],$  the Ritz projection 
$w_h(t): H^1_0(\Omega)\rightarrow \mathcal{S}_h$ is defined by
\begin{equation}\label{eq2.9}
        (\nabla \eta(t), \nabla v_h)=0,  \quad  \text{for any} \quad v_h\in \mathcal{S}_h, ~~{\rm where}~~\eta(t) = u(t) - w_h(t).
   \end{equation}

Following the proof of the regularity property in  Theorem \ref{theorem2.1}, we obtain the stability of the semi-discrete numerical solution in the next theorem.
\begin{theorem}\label{theorem2.2}
    Under the assumptions of Theorem \ref{theorem2.1}, the  solution $u_h$ of  \eqref{eq2.6} satisfies
\begin{equation*}
     \begin{array}{ll}
          \|u'_h(t)\|  + \|\nabla u_h(t)\|   \leq C\left( \|u_{1h}\|+ \|\nabla u_{0h}\|  +
        \int_0^{t} \|f(s)\|ds\right), \text{ for } 0<t\leq T.
     \end{array}
     \end{equation*}
\end{theorem}
In the next theorem, we show the convergence of the semi-discrete Galerkin finite element  solution. For convenience, let $E(t)=u(t)-u_h(t)$.
\begin{theorem}\label{theorem2.3} Let $u_0\in H^2(\Omega)$, $u_1\in H^1(\Omega)$, and $f,f'\in L^1(0,T;L^2(\Omega))$.
Under the assumptions $(\mathbf{A1})$-$(\mathbf{A2})$, and for  $0<t\leq T$, we have
\begin{equation*}
    \left\|E'(t)\right\|^2  
    + \left\|\nabla E(t)\right\|^2                 \leq  Ch^2\int_0^{t} \left( \|u''(s)\|^2_1  + \| u'(s)\|^2_1 + \| u(s)\|^2_2  \right)ds +Ch^2\|u(t)\|_1^2.
   \end{equation*}
\end{theorem}
\begin{proof} First, recall the definition of $\eta$  in \eqref{eq2.9}. The following Ritz projection error estimate is well known  \cite{Thomee}: for $r=0,1$, and with $\partial_t$ denotes the partial time derivative, 
\begin{equation}\label{eq2.10}
   \begin{array}{ll}
      \dsp  \|\partial_t^\theta\eta(t)\|+h\|
\nabla(\partial_t^\theta\eta)(t)\| \leq Ch^{r+1}\|\partial_t^\theta u(t)\|_{H^{r+1}(\Omega)},  \quad \theta=0,1,2.
   \end{array}
   \end{equation}
   To derive the error bound in Theorem \ref{theorem2.3}, we decompose $E$ as: 
\begin{equation}\label{eq2.11}
   \begin{split}
        & E(t)=u(t)-u_h(t)=[u(t)-w_h(t)]-[u_h(t)-w_h(t)]:=\eta(t)-\xi(t), 
   \end{split}
\end{equation}
and the main task is to estimate  $\xi(t)$. We use \eqref{eq2.6}, \eqref{eq2.3} and  \eqref{eq2.9} to get
\begin{align*}
        (\xi''(t), v_h) &+ \left( G_0(u_h, \nabla u_h)u_h'(t) - G_0(w_h, \nabla w_h) w_h'(t), v_h \right) 
         \nonumber\\  & + (\nabla \xi (t), \nabla v_h)   -  \int_{0}^{t}\beta(t-s)(\nabla \xi(s), \nabla v_h) ds \\
        \nonumber& =  (\eta''(t), v_h) + \left( G_0(u, \nabla u)u'(t) - G_0(w_h, \nabla w_h) w_h'(t), v_h \right), 
   \end{align*}
for any $v_h \in \mathcal{S}_h$, where $G_0(y,z)$ is defined in \eqref{eq4.16}. Choosing $v_h=\xi'(t)$ gives
\begin{equation}\label{eq5.1}
   \begin{split}
       \frac{1}{2}\frac{d}{dt}& \|\xi'(t)\|^2   + \left( G_0(u_h, \nabla u_h)u_h'(t) - G_0(w_h, \nabla w_h) w_h'(t), \xi'(t) \right) 
         \\  & + \frac{1}{2}\frac{d}{dt}\|\nabla \xi (t)\|^2   +  \int_{0}^{t}K'(t-s)(\nabla \xi(s), \nabla \xi'(t)) ds \\
        & =  (\eta'', \xi'(t))+ \left( G_0(u, \nabla u)u'(t) - G_0(w_h, \nabla w_h) w_h'(t), \xi'(t) \right).
   \end{split}
   \end{equation}
Denoting \begin{equation}\label{qqq1}
   \varphi_{\vartheta}(t):=  \mu_1\| \xi(t) \vartheta + w_h(t)\|^2+\mu_2\|\nabla \xi(t)\vartheta+\nabla w_h(t)\|^2,
   \end{equation}   
and so,  the second term on the left-hand side of \eqref{eq5.1} can be rewritten as:
\begin{equation}\label{eq5.2}
   \begin{split}
       B(t):&=\big( G_0(u_h, \nabla u_h)u_h'(t) - G_0(w_h, \nabla w_h) w_h'(t), \xi'(t) \big) \\
       & = \left(\int_0^1 \frac{d}{d\vartheta} \left[G(\varphi_{\vartheta}(t))(\vartheta\xi'(t) + w_h'(t) )\right]  d\vartheta, \xi'(t)\right) \\
       & = \int_0^1 G'(\varphi_{\vartheta}(t)) \frac{d \varphi_{\vartheta}(t)}{d \vartheta} \left[ \vartheta\|\xi'(t)\|^2 + (w_h', \xi'(t))\right] d\vartheta 
       + \int_0^1 G(\varphi_{\vartheta}(t)) d\vartheta \|\xi'(t)\|^2,
   \end{split}
   \end{equation}
where
\begin{equation*}
    \frac{d \varphi_{\vartheta}(t)}{d \vartheta}=2\vartheta \left(\mu_1\| \xi(t)\|^2+\mu_2\|\nabla\xi(t)\|^2 \right) + 2\left[ \mu_1(w_h(t), \xi(t)) +  \mu_2(\nabla w_h(t), \nabla \xi(t))\right],
\end{equation*}
and
$ \int_0^1 G(\varphi_{\vartheta}(t)) d\vartheta \|\xi'(t)\|^2 \geq g_0 \|\xi'(t)\|^2$. Then we have
\begin{equation}\label{eq5.3}
   \begin{split}
       B(t) \geq  2(J_1(t; \vartheta, w_h)+J_2(t; \vartheta, w_h)) + g_0 \|\xi'(t)\|^2,  \quad t>0,
   \end{split}
   \end{equation}
where 
\begin{equation}\label{eq5.4}
   \begin{split}
      & J_1(t; \vartheta, w_h) = \int_0^1 G'_{\vartheta}  \left(\mu_1\| \xi(t)\|^2+\mu_2\|\nabla\xi(t)\|^2 \right) [\vartheta^2\|\xi'(t)\|^2+ \vartheta(w_h'(t), \xi'(t))] d\vartheta, \\
      & J_2(t; \vartheta, w_h) = \int_0^1 G'_{\vartheta}\Big( \mu_1(w_h(t), \xi(t)) +  \mu_2(\nabla w_h(t), \nabla\xi(t))\Big)[\vartheta \|\xi'(t)\|^2 +(w_h'(t), \xi'(t))] d\vartheta,
   \end{split}
   \end{equation}
with $G'_{\vartheta}=G'(\varphi_{\vartheta}(t))$. Substitute \eqref{eq5.2} and \eqref{eq5.3} into \eqref{eq5.1}, and use the inequality  
\begin{equation}\label{eq5.5}
        (\eta'', \xi'(t)) 
         \leq  \|\eta''\| \|\xi'(t)\| \leq \frac{g_0}{8}\|\xi'(t)\|^2 +  \frac{2}{g_0}\|\eta''\|^2,
   \end{equation} 
we conclude that 
\begin{multline}\label{eq5.6}
       \frac{1}{2} \frac{d}{dt}\|\xi'(t)\|^2  + \frac{7g_0}{8}  \|\xi'(t)\|^2 + \frac{1}{2}\frac{d}{dt}\|\nabla \xi (t)\|^2 
         \\   +  \int_{0}^{t}K'(t-s)(\nabla \xi(s), \nabla \xi'(t)) ds\leq \frac{2}{g_0}\|\eta''\|^2 + 2\sum\limits_{q=1}^3 \left|J_q(t; \vartheta, w_h)\right|,
   \end{multline} 
   where  $$J_3(t; \vartheta, w_h)=\left( G_0(u, \nabla u)u'(t) - G_0(w_h, \nabla w_h) w_h'(t), \xi'(t) \right).$$
The next task is to  estimate the terms $\left|J_q(t; \vartheta, w_h)\right|$ for $q=1,2,3$.
From the definition of $w_h$,   $\|\nabla w_h'(t)\|\le \|\nabla u'(t)\|$, and so, using  the Poincar\'{e} inequality and \eqref{eq2.5n},  we get 
\begin{equation}\label{eq5.7}
      \|w_h'(t)\|  \le   C\|\nabla u'(t)\|  
       \leq C\left( \|\nabla u_1\| + \|\Delta u_0\| + \|f(0)\|  +  \int_{0}^{t}(\|f(s)\|+\|f'(s)\|)ds \right), 
   \end{equation} 
which, together with Theorem \ref{theorem2.2} and the Poincar\'{e} inequality, gives
\begin{equation}\label{eq5.8}
       \|\xi'(t)\|   \leq \|w_h'(t)\| + \|u_h'(t)\|  \leq C\left( \|\nabla u_1\| + \|\Delta u_0\| + \|f(0)\| + \int_0^t (\|f(s)\|+\|f'(s)\|)ds \right).
   \end{equation} 
Using  \eqref{eq5.7}, \eqref{eq5.8}, and  the Poincar\'{e} inequality,  we obtain
\begin{equation}\label{eq5.9}
       \left|J_1(t; \vartheta, w_h)\right| \leq C(\| \xi(t)\|^2 + \|\nabla \xi(t)\|^2)\leq C\|\nabla \xi(t)\|^2. 
   \end{equation}
From the definition of $w_h$,   $\|\nabla w_h(t)\|\le \|\nabla u(t)\|$,  and thus, using \eqref{eq2.5},  we notice that
\begin{equation}\label{eq5.11}
       \|\nabla w_h(t)\|  \leq C\left( \|\nabla u_1\|+ \|\nabla u_0\| +
        \int_{0}^{t}\|f(s)\|ds \right),
   \end{equation}
and hence, using this, in addition to  the estimates in \eqref{eq5.7} and \eqref{eq5.8}, and \eqref{eq5.11},  and the Young's inequality yield 
\begin{equation*}
       \left|J_2(t; \vartheta, w_h)\right| \leq C(\|\nabla \xi(t)\| + \|\xi(t)\|) \|\xi'(t)\| \leq  C\|\nabla \xi(t)\| \|\xi'\| 
        \leq \frac{g_0}{4}\|\xi'(t)\|^2 + C\|\nabla \xi(t)\|^2.
   \end{equation*}
To bound $\left|J_3(t; \vartheta, w_h)\right|$, we apply the assumptions ($\mathbf{A1}$)-($\mathbf{A2}$) and the Cauchy-Schwarz inequality to get
\begin{equation*}
   \begin{split}
       J_3(t; \vartheta, w_h)
       &= G_0(u,\nabla u) (u'- w_h',  \xi') + ( [G_0(u,\nabla u)-G_0(w_h, \nabla w_h)]w_h',  \xi') \\
       & \leq 
       g_1 \|\eta'\| \|\xi'(t)\|  +
       \mathcal{L} \left|  \mathcal{Q}[u,\nabla u](t) - \mathcal{Q}[w_h,\nabla w_h](t)  \right| \|w_h'\| \|\xi'(t)\|,
   \end{split} 
   \end{equation*}
where $\mathcal{Q}[y,z](t):=\mu_1\|y(t)\|^2+\mu_2\|z(t)\|^2$. Then we utilize the inequality $\|y\|^2-\|z\|^2\leq (\|y\|+\|z\|)\|y-z\|$, \eqref{eq5.7}, \eqref{eq5.11} and Young's inequality to get
\begin{equation}\label{eq5.14}
   \begin{split}
       |J_3(t; \vartheta, w_h)|
       & \leq \frac{g_0}{8}\|\xi'(t)\|^2 + \frac{2g_1}{g_0}\|\eta'\|^2 + C(  \|\eta(t)\|^2 + \|\nabla\eta(t)\|^2).
   \end{split} 
   \end{equation}
Now we invoke \eqref{eq5.9}--\eqref{eq5.14} in \eqref{eq5.6} to obtain
\begin{multline}\label{eq5.15}
       \frac{1}{2}\frac{d}{dt}\|\xi'(t)\|^2  + \frac{g_0}{2}  \|\xi'(t)\|^2 + \frac{1}{2}\frac{d}{dt}\|\nabla \xi (t)\|^2     
          +  \int_{0}^{t}K'(t-s)(\nabla \xi(s), \nabla \xi'(t)) ds 
         \\
         \leq \frac{2}{g_0}\|\eta''\|^2   + \frac{2g_1}{g_0}\|\eta'\|^2 + C \|\nabla\eta\|^2 + C\|\nabla \xi(t)\|^2.
   \end{multline} 
We incorporate $\xi'(0) = \xi(0)=0$ to integrate \eqref{eq5.15} over the time interval $(0,t)$ and apply
   \begin{multline*}
       \Phi(t; \xi) 
        \geq K(t)\left( \nabla \xi(0), \nabla \xi(t) \right) + \int_{0}^{t}\beta(s)(\nabla \xi(0),\nabla \xi(s)) ds  
         \\
         - \frac{K_0}{2}\left[\|\nabla \xi(t)\|^2 + \|\nabla \xi(0)\|^2\right] = -\frac{K_0}{2}\|\nabla \xi(t)\|^2
  \end{multline*}
to get
\[       \frac{1}{2}\|\xi'(t)\|^2  + \frac{g_0}{2}  \int_0^{t}\|\xi'\|^2ds + \frac{\mu_0}{2}\|\nabla \xi (t)\|^2     
         \leq  C \int_0^{t} \left(\|\eta''\|^2 + \|\eta'\|^2 +   \|\nabla\eta\|^2 \right)ds  + C \int_0^{t}\|\nabla \xi\|^2 ds.\]
Apply the Gr\"{o}nwall's lemma to get
\[      \|\xi'(t)\|^2  + g_0\int_0^{t}\|\xi'\|^2ds + \mu_0\|\nabla \xi (t)\|^2 
         \leq  C \int_0^{t} \left(\|\eta''\|^2 + \|\eta'\|^2  +   \|\nabla\eta\|^2 \right)ds,\]
which, together with \eqref{eq2.10}, gives
\[ \|\xi'(t)\|^2   + \mu_0\|\nabla \xi (t)\|^2  
         \leq  Ch^2 \int_0^{t} \left( \|u''\|^2_1  + \| u'\|^2_1  + \| u\|^2_2  \right)ds. 
   \]
We combine this and \eqref{eq2.10} to complete the proof.
\end{proof}

\section{Fully-discrete scheme} \label{sec4}
In this section, we shall establish and analyze a fully discrete scheme.
Using $K'(t)=-\beta(t)$ and denoting $\mu_0:=1-K_0$, we have
\begin{equation*}
    \int_0^t K'(t-s)\Delta u(s)ds = K(t)\Delta u_0 - K_0 \Delta u(t) + \int_0^t K(t-s)\Delta u'(s) ds.
\end{equation*}
Thus we reformulate \eqref{eq1.1} as
\begin{align}
    & u''(t) + q(t)\,u'(t)  - \mu_0 \Delta u(t) - \int_0^t K(t-s)\Delta u'(s) ds = f(t) + K(t)\Delta u_0. \label{ModelA}
\end{align}
The task now is to   propose and analyze a fully discrete scheme for solving problem \eqref{ModelA}.
Let $\tau$ be the temporal step size and $U^n\in \mathcal{S}_h$ with $0\leq n\in\mathbb N$ be the approximation of $u(t_n)$ with $t_n=n\tau$. Let  $\delta_tU^n=(U^n-U^{n-1})/\tau$, $\bar{\delta}_tU^n=(U^{n+1}-U^{n-1})/(2\tau)$, $\delta_t^2 U^n=\delta_t(\delta_tU^{n+1})$, and $\widetilde{U}^n=(U^{n+1}+U^{n-1})/2$, $n\geq 1$. Using a linear polynomial interpolation
\begin{align*}
    \varphi(s)\approx \mathcal{N}_1(s):=\varphi(t_{p-1})+\frac{s-t_{p-1}}{\tau}[\varphi(t_{p})-\varphi(t_{p-1})],\quad{\rm for}~~s \in [t_{p-1},t_p]\,,
\end{align*}   we  approximate the integral term $\int_{0}^{t_n}K(t_n-s)\varphi(s)ds$  by  
\begin{equation}\label{eq2.14}
        Q_n(\varphi) = \sum\limits_{p=1}^{n}\int_{t_{p-1}}^{t_p}K(t_n-s)\mathcal{N}_1(s)ds= \sum\limits_{p=0}^{n}\widetilde{\kappa}_{np}\varphi(t_p), \quad n\geq 1,
   \end{equation}
   where  
   \begin{equation}\label{coeff}
   \begin{split} 
        \widetilde{\kappa}_{np} = \int_{-\min(\tau, t_p)}^{\min(\tau, t_{n-p})} K(t_n-s) \max\left( 1-\left| \frac{s}{\tau} \right|, 0 \right) dt,  \quad 0\leq p \leq n.
   \end{split} 
   \end{equation}
Using  \eqref{eq2.14}  and the notations  $u^n:=u(t_n)$ and $f^n:=f(t_n)$, we write \eqref{ModelA} at $t_n$ as 
\begin{align}
    &\delta_t^2 u^n  + q(t_n)\bar{\delta}_tu^n - \mu_0 \Delta \widetilde{u}^n - \sum\limits_{p=0}^{n}\widetilde{\kappa}_{np}\Delta \bar{\delta}_t u^p = f^n + K(t_n)\Delta u_0 + \mathcal{R}^n, \label{ModelC}\\
    & \delta_t u^1  = u_1 + \frac{\tau}{2}u_2 + \mathcal{R}^0, \quad u^0 = u_0=u(0),~u_1=u'(0),~u_2=u''(0), \label{ModelE}
\end{align}
in which $\bar{\delta}_t u^0 = u_1$,  $\mathcal{R}^n:=\sum\limits_{l=1}^{4}R_l^n$ with $n\geq 1$, $u_2=u''(0)=-q(0)u_1 + \Delta u_0 + f^0$, and 
\begin{align}
    &R_1^n = \delta_t^2 u^n - u''(t_n),\quad  R_2^n = q(t_n)\left[ \bar{\delta}_tu^n - u'(t_n) \right],\quad R_3^n = \mu_0\left[ \Delta \widetilde{u}^n - \Delta u^n \right], \label{qq1} \\
    &  R_4^n = \int_0^{t_n} K(t_n-s)\Delta u'(s) ds - \sum\limits_{p=0}^{n}\widetilde{\kappa}_{np}\Delta  u'(t_p) + \sum\limits_{p=1}^{n}\widetilde{\kappa}_{np}\Delta  (u'(t_p) - \bar{\delta}_tu^p), \label{qq4} \\
    & \mathcal{R}^0 = \frac{1}{2\tau}\int_0^{\tau} (\tau - s)^2 u'''(s)dt. \label{qq5}
\end{align}
The weak form of \eqref{ModelC} is defined as: for any $\psi\in H_0^1(\Omega)$ and for $n\ge 1,$  
   \begin{multline}\label{qq6}
        (\delta_t^2u^n,\psi) + G\left(   \mu_1\|u^n\|^2+\mu_2\|\nabla u^n\|^2 \right) (\bar{\delta}_tu^n,\psi) 
         +  \mu_0(\nabla \widetilde{u}^n,\nabla \psi) \\ + \sum\limits_{p=0}^{n}\widetilde{\kappa}_{np}(\nabla \bar{\delta}_t u^p,\nabla \psi)
          = (f^n, \psi) + K(t_n)(\nabla u^0, \nabla\psi) + (\mathcal{R}^n, \psi).
   \end{multline}
We drop the truncation errors and define then the fully discrete Galerkin scheme as: for $n\ge 1,$ 
   \begin{multline}\label{eq2.15}
        (\delta_t^2U^n,\psi_h) + G\left(   \mu_1\|U^n\|^2+\mu_2\|\nabla U^n\|^2 \right) (\bar{\delta}_tU^n,\psi_h) 
         +  \mu_0(\nabla \widetilde{U}^n,\nabla \psi_h)\\
         + \sum\limits_{p=0}^{n}\widetilde{\kappa}_{np}(\nabla \bar{\delta}_t U^p,\nabla \psi_h)
          = (f^n, \psi_h) + K(t_n)(\nabla U^0, \nabla\psi_h),
   \end{multline}
for any  $\psi_h\in \mathcal{S}_h$, 
given the initial values
\begin{equation}\label{eq2.16}
    U^0=u_{0h}, \quad U^1=u_{0h}+ \tau  u_{1h}+\frac{\tau^2}{2}u_{2h},
\end{equation}
where $u_{jh}$, $j=0,1,2$ are suitable approximations of $u_j$ in $\mathcal{S}_h$.

\subsection{Stability analysis}
We show the stability  of the fully discrete solution. 
\begin{theorem}\label{theorem2.4}
Suppose that $\beta(t)$ is given by \eqref{eq1.5} and the assumptions $(\mathbf{A1})$-$(\mathbf{A2})$ hold with $\tau\sum_{i=0}^n\|f^i\|<\infty$ for $0<T< \infty$. Then for any $1\leq m\leq T/\tau$, it holds 
\begin{equation*}
       \left\|U^m\right\|_A  \leq    \left\|U^0\right\|_A  + C\left( \|\nabla u_{0h}\| + \|\nabla u_{1h}\| + \tau\sum_{i=0}^n\|f^i\| \right),
   \end{equation*}
where $C$ is independent from $T$, $m$ and $\tau$, and the semi-norm $\|\cdot\|_A$ is defined as
 \begin{equation}\label{eq6.3}
   \begin{split}
        \|U^m\|_A = \sqrt{ \|\delta_tU^{m+1}\|^2 + \frac{\mu_0}{2} \left(\|\nabla U^{m+1}\|^2+\|\nabla U^{m}\|^2 \right) }. 
   \end{split}
   \end{equation}
   Consequently, the fully-discrete scheme (\ref{eq2.15}) has a unique solution.
\end{theorem}
\begin{proof} Choose $\psi=\bar{\delta}_tU^n$ in \eqref{eq2.15} to obtain
   \begin{multline}\label{eq6.1}
        (\delta_t^2U^n,\bar{\delta}_tU^n) + G\left(   \mu_1\|U^n\|^2+\mu_2\|\nabla U^n\|^2 \right) (\bar{\delta}_tU^n,\bar{\delta}_tU^n) 
         +  \mu_0(\nabla \widetilde{U}^n,\nabla \bar{\delta}_tU^n)\\  + \sum\limits_{p=0}^{n}\widetilde{\kappa}_{np}(\nabla \bar{\delta}_t U^p,\nabla \bar{\delta}_tU^n)
          = (f^n, \bar{\delta}_tU^n) + K(t_n)(\nabla U^0, \nabla\bar{\delta}_tU^n), \quad n\geq 1,
   \end{multline}
where, by direct calculations, 
   \begin{equation}\label{eq6.2}
   \begin{split}
        (\delta_t^2U^n,\bar{\delta}_tU^n) &= \frac{1}{2 \tau} \left( \|\delta_tU^{n+1}\|^2 -\|\delta_tU^{n}\|^2   \right), \\
         (\nabla \widetilde{U}^n,\nabla \bar{\delta}_tU^n) &= \frac{1}{4\tau} \left(\|\nabla U^{n+1}\|^2+\|\nabla U^{n}\|^2 \right) - \left(\|\nabla U^{n}\|^2+\|\nabla U^{n-1}\|^2 \right).
   \end{split}
   \end{equation}
We sum \eqref{eq6.1} multiplied by $2\tau$ from 1 to $m$ and apply the new norm $\|\cdot\|_A$ defined by \eqref{eq6.3}, the assumption ($\mathbf{A1}$) and \eqref{eq6.2} to get
   \begin{multline}\label{eq6.4}
        \|U^m\|_A^2  + 2g_0\tau \sum\limits_{n=1}^{m}\|\bar{\delta}_tU^n\|^2  + 2 \tau \sum\limits_{n=1}^{m}\sum\limits_{p=0}^{n}\widetilde{\kappa}_{np}(\nabla \bar{\delta}_tU^p,\nabla \bar{\delta}_tU^n)  \\
         \leq \|U^0\|_A^2 +  2 \tau \sum\limits_{n=1}^{m}K(t_n)(\nabla U^0, \nabla \bar{\delta}_tU^n) + 2 \tau \sum\limits_{n=1}^{m}(f^n, \bar{\delta}_tU^n).
   \end{multline}
Since
\begin{equation*}
    \begin{split}
        2 \tau \sum\limits_{n=1}^{m}K(t_n)\nabla \bar{\delta}_tU^n
        & =  \sum\limits_{n=1}^{m}K(t_n)\left[ (\nabla U^{n+1}+\nabla U^{n}) - (\nabla U^{n} + \nabla U^{n-1}) \right] \\
        & = K(t_m)\left(\nabla U^{m+1}+\nabla U^{m}\right) - K(t_1)(\nabla U^{1}+\nabla U^{0}) \\
        & \quad + \sum\limits_{n=1}^{m-1}\left[ K(t_n)-K(t_{n+1}) \right] (\nabla U^{n+1}+\nabla U^{n}),
    \end{split}
\end{equation*}
\[2 \tau \sum\limits_{n=1}^{m}K(t_n)(\nabla U^0, \nabla \bar{\delta}_tU^n) 
        \leq C\|\nabla U^0\|\left( \left\|U^0\right\|_A  + \left\|U^m\right\|_A \right)   
         + C\|\nabla U^0\|\sum\limits_{n=1}^{m-1} \int_{t_{n}}^{t_{n+1}} |\beta(s)|ds \left\|U^n\right\|_A.\]
Following \cite[51--52]{McLean}, it holds
\[
        \tau \sum\limits_{n=1}^{m}\sum\limits_{p=0}^{n}\widetilde{\kappa}_{np}(\nabla \bar{\delta}_tU^p,\nabla \bar{\delta}_tU^n) \geq \tau \sum\limits_{n=1}^{m} \widetilde{\kappa}_{n0} (\nabla u_{1h},\nabla \bar{\delta}_tU^n).
 \]
Furthermore, recall the definition of $\|\cdot\|_A$ in \eqref{eq6.3},  thus we have
\begin{align}
     \|\bar{\delta}_tU^n\| &\leq \frac{\|U^n\|_A+\|U^{n-1}\|_A}{2}, \nonumber \\
     \left| \tau (\nabla u_{1h},\nabla \bar{\delta}_tU^n) \right| &\leq \sqrt{\frac{2}{\mu_0}}  \frac{(\|U^n\|_A+\|U^{n-1}\|_A)}{2} \|\nabla u_{1h}\|. \label{eq6.7}
\end{align}
We substitute the above contribution   into \eqref{eq6.4} to get
   \begin{equation*}
   \begin{split}
        \|U^m\|_A^2  &\leq \|U^0\|_A^2  + 2\sqrt{\frac{2}{\mu_0}} \sum\limits_{n=1}^{m} \widetilde{\kappa}_{n0}  \frac{(\|U^n\|_A+\|U^{n-1}\|_A)}{2} \|\nabla u_{1h}\| \\
        & \quad + C\|\nabla U^0\|\left( \left\|U^0\right\|_A  + \left\|U^m\right\|_A \right)  + C\|\nabla U^0\|\sum\limits_{n=1}^{m-1} \int_{t_{n}}^{t_{n+1}} |\beta(s)|ds \left\|U^n\right\|_A \\
        & \quad + 2 \tau \sum\limits_{n=1}^{m}\|f^n\| \frac{\|U^n\|_A+\|U^{n-1}\|_A}{2}.
   \end{split}
   \end{equation*}
We choose a suitable $\ell$ such that $\|U^\ell\|_A = \max\limits_{0\leq n \leq m}\|U^n\|_A$, which, together with $\beta\in L_{1}(0,\infty)$, yields
   \begin{equation}\label{eq6.9}
   \begin{split}
         \|U^\ell\|_A   
         & \leq \|U^0\|_A  + C\left[ \sum\limits_{n=1}^{m} \widetilde{\kappa}_{n0}   \|\nabla u_{1h}\| + \tau \sum\limits_{n=1}^{m}\|f^n\|  + \|\nabla u_{0h}\|  \right].
   \end{split}
   \end{equation}
By \eqref{coeff} and the monotonicity of $s^{\alpha-1}$, we have for $m\geq 1$ and $\sigma \geq 1$
   \begin{equation}\label{eq6.10}
   \begin{split}
         \sum\limits_{n=1}^{m} \widetilde{\kappa}_{n0} & \leq \sum\limits_{n=1}^{m} \int_{t_{n-1}}^{t_n} K(t) dt = \int_{0}^{t_m} K(t) dt \leq \int_{0}^{t_m} \int_t^{\infty }|\beta(s)| ds dt \\
        & \leq  \int_0^{\infty}  \int_t^{\infty}\frac{s^{\alpha-1}}{\Gamma(\alpha)} e^{-\sigma s}dsdt \leq    \int_0^{\infty} \frac{t^{\alpha-1}}{\Gamma(\alpha)} \int_t^{\infty}e^{-\sigma s}dsdt  
        \\ &= \int_0^{\infty} \frac{t^{\alpha-1}}{\Gamma(\alpha)} \frac{e^{-\sigma t}}{\sigma}dt 
     \leq \int_0^{\infty} \frac{e^{-t}t^{\alpha-1}}{\Gamma(\alpha)}dt = 1.
   \end{split}
   \end{equation}
Combine \eqref{eq6.9} and \eqref{eq6.10} to complete the proof.

\end{proof}

\subsection{Error estimate}
We  derive error estimate of the fully discrete Galerkin scheme.

\begin{theorem}\label{theorem2.5}
Suppose that $\beta(t)$ is given by \eqref{eq1.5} and the assumptions $(\mathbf{A1})$-$(\mathbf{A2})$ hold with  $u_0\in H^2(\Omega)$, $u_1\in H^1(\Omega)$, and $f,f'\in L^1(0,T; L^2(\Omega))$ for $0<T<\infty$. Then the following error estimate holds for any  $1\le m\leq T/\tau$:
\begin{align}
        \left\|\nabla \left(U^{m} - u(t_{m})\right) \right\| &\leq Ch \left( \|u'\|_{L^{\infty}(H^1) } +  \|u_1\|_1 +\tau \|u_0\|_2 \right) \nonumber \\
        & + Ch\left(\|u\|_{L^{\infty}(H^2)}+\|u'\|_{L^{\infty}(H^1)}+\|u''\|_{L^{\infty}(H^1)}\right)  \nonumber \\
        & + C\tau \int_0^{2\tau} \|u'''(t)\|dt + C\tau^2 \int_{\tau}^{t_{m+1}} \|u^{(4)}(t)\|dt \nonumber \\
        & +  C\tau^2 \int_0^{t_{m+1} }\|u'''(t)\|dt + C\tau^2 \int_0^{t_{m+1} }\|\Delta u''(t)\|dt \nonumber \\
        & + C \tau \int_0^{\tau} \|\Delta u''(t)\|dt + C\tau^2  \int_{\tau}^{t_{m+1}} \|\Delta u'''(t)\|dt. \label{eq2.20}
    \end{align}
\end{theorem}
\begin{proof} We split $u(t_n)-U^n=[u(t_n)-w_h(t_n)]-[U^n-w_h(t_n)]:=\eta^n-\xi^n$ for $n\geq 1$ where $w_h(t)$ is the elliptic projection of $u(t_n)$ in \eqref{eq2.9}, and it suffices to bound $\|\xi^n\|$. We subtract \eqref{qq6} from \eqref{eq2.15} to get
   \begin{equation*}
   \begin{split}
        (\delta_t^2\xi^n,\psi) &+ \left[G_0\left(   U^n, \nabla U^n\right) (\bar{\delta}_tU^n,\psi) - G_0\left( w_h(t_n),\nabla w_h(t_n)\right) (\bar{\delta}_tw_h(t_n),\psi)  \right]  \\
        &+ \mu_0 (\nabla \widetilde{\xi}_1^n,\nabla \psi)  + \sum\limits_{p=1}^{n}\widetilde{\kappa}_{np}(\nabla \bar{\delta}_t\xi^p,\nabla \psi) \\
        & = (\delta_t^2\eta^n,\psi)   + \left(G_0\left(u^n,\nabla u^n\right) \bar{\delta}_tu^n - G_0\left( w_h(t_n),\nabla w_h(t_n)\right) \bar{\delta}_tw_h(t_n), \psi\right) \\
        & + (  G_0\left(u^n,\nabla u^n\right)(u'(t_n) - \bar{\delta}_tu^n)  , \psi) + \sum\limits_{j=1}^{4} ( R_j^n, \psi), \quad n\geq 1,
   \end{split}
   \end{equation*}
for $\psi\in \mathcal{S}_h$. Choosing $\psi=\bar{\delta}_t\xi^n$ gives
   \begin{equation}\label{eq7.1}
   \begin{split}
(\delta_t^2\xi^n,\bar{\delta}_t\xi^n) &+ \left[G_0\left(   U^n, \nabla U^n\right) (\bar{\delta}_tU^n,\bar{\delta}_t\xi^n) - G_0\left( w_h(t_n),\nabla w_h(t_n)\right) (\bar{\delta}_tw_h(t_n),\bar{\delta}_t\xi^n)  \right]  \\
        &+ \mu_0 (\nabla \widetilde{\xi}_1^n,\nabla \bar{\delta}_t\xi^n)  + \sum\limits_{p=1}^{n}\widetilde{\kappa}_{np}(\nabla \bar{\delta}_t\xi^p,\nabla \bar{\delta}_t\xi^n) \\
        & = (\delta_t^2\eta^n,\bar{\delta}_t\xi^n)   + \left(G_0\left(u^n,\nabla u^n\right) \bar{\delta}_tu^n - G_0\left( w_h(t_n),\nabla w_h(t_n)\right) \bar{\delta}_tw_h(t_n), \bar{\delta}_t\xi^n \right) \\
        & + (  G_0\left(u^n,\nabla u^n\right)(u'(t_n) - \bar{\delta}_tu^n), \bar{\delta}_t\xi^n) + \sum\limits_{j=1}^{4} ( R_j^n, \bar{\delta}_t\xi^n ), \quad n\geq 2,
   \end{split}
   \end{equation}
We first find a lower bound of  the second term of the left-hand side of \eqref{eq7.1}. By the Newton-Leibniz formula, we rewrite this term as
\begin{equation}\label{eq7.2}
   \begin{split}
       B_0(n; \tau):&=\left( G_0(U^n, \nabla U^n)\bar{\delta}_tU^n - G_0(w_h(t_n), \nabla w_h(t_n)) \bar{\delta}_tw_h(t_n), \bar{\delta}_t\xi^n \right) \\
       & = \left(\int_0^1 \frac{d}{d\vartheta} \left[G(\varphi_{\vartheta}(t_n))(\vartheta \bar{\delta}_t\xi^n + \bar{\delta}_tu_{h}^{*}(t_n) )\right]  d\vartheta, \bar{\delta}_t\xi^n \right) \\
       & = \int_0^1 G'(\varphi_{\vartheta}(t_n)) \frac{d \varphi_{\vartheta}(t_n)}{d \vartheta} \left[ \vartheta\|\bar{\delta}_t\xi^n\|^2 + \left(\bar{\delta}_tw_h(t_n), \bar{\delta}_t\xi^n\right)\right] d\vartheta \\
       &+ \int_0^1 G(\varphi_{\vartheta}(t_n)) d\vartheta \|\bar{\delta}_t\xi^n\|^2,
   \end{split}
   \end{equation}
where $\varphi_{\vartheta}(t)$ is given by \eqref{qqq1}
and $$\frac{d \varphi_{\vartheta}(t_n)}{d \vartheta}:=2\vartheta \left[\mu_1\| \xi^n\|^2+\mu_2\|\nabla\xi^n\|^2 \right] + 2\left[ \mu_1(w_h(t_n), \xi^n) +  \mu_2(\nabla w_h(t_n), \nabla\xi^n) \right].$$
 We then use $ \int_0^1 G(\varphi_{\vartheta}(t_n)) d\vartheta \|\bar{\delta}_t\xi^n\|^2 \geq g_0 \|\bar{\delta}_t\xi^n\|^2$ to get
\begin{equation*}
   \begin{split}
      B_0(n; \tau) \geq  2\sum\limits_{q=1}^4 \hat{J}_q(n; \vartheta, \tau) + g_0 \|\bar{\delta}_t\xi^n\|^2,  \quad t>0,
   \end{split}
   \end{equation*}
where 
\begin{equation}\label{eq7.4}
   \begin{split}
      & \hat{J}_1(n; \vartheta, \tau) = \int_0^1 G'_{\vartheta n}  \left(\mu_1\| \xi^n\|^2+\mu_2\|\nabla\xi^n\|^2 \right) [\vartheta^2\|\bar{\delta}_t\xi^n\|^2+\vartheta(\bar{\delta}_tw_h(t_n), \bar{\delta}_t\xi^n)] d\vartheta, \\
      & \hat{J}_2(n; \vartheta, \tau) = \int_0^1 G'_{\vartheta n}\Big( \mu_1(w_h(t_n), \xi^n) +  \mu_2(\nabla w_h(t_n), \xi^n)\Big)[\vartheta \|\bar{\delta}_t\xi^n\|^2 +\bar{\delta}_tw_h(t_n), \bar{\delta}_t\xi^n)] d\vartheta,
   \end{split}
   \end{equation}
with $G'_{\vartheta n}=G'(\varphi_{\vartheta}(t_n))$. Below, we shall analyze the terms $\hat{J}_1$ and $\hat{J}_2$ in \eqref{eq7.4}. First, we apply Theorem \ref{theorem2.4} to obtain $\|\bar{\delta}_tU^n\|\leq C$, and  \eqref{eq5.7} gives
\begin{equation}\label{eq7.5}
   \begin{split}
       \|\bar{\delta}_t w_h(t_n)\| \leq \Big\| \frac{1}{2\tau} \int_{t_{n-1}}^{t_{n+1}} w_h'(t)dt \Big\| \leq \frac{C}{\tau}  \int_{t_{n-1}}^{t_{n+1}}\|w_h'(t)\|dt \leq C.
   \end{split}
   \end{equation}
Then we employ the triangle inequality to get
\begin{equation}\label{eq7.6}
   \begin{split}
       \|\bar{\delta}_t\xi^n\|\leq \|\bar{\delta}_tU^n\|+\|\bar{\delta}_t w_h(t_n)\| \leq C, \quad  n\geq 2. 
   \end{split}
   \end{equation}
Based on \eqref{eq7.5}--\eqref{eq7.6} and the assumption ($\mathbf{A2}$), we have for $n\geq 2$
\[\|\hat{J}_1(n; \vartheta, \tau)\| \leq  C\left(\| \xi^n\|^2+\|\nabla \xi^n\|^2 \right) \leq C\|\nabla \xi^n\|^2.\]
Then we use \eqref{eq7.5}--\eqref{eq7.6}, Poincar\'{e} inequality and Young inequality to get
\[\|\hat{J}_2(n; \vartheta, \tau)\| \leq C\left(\| \nabla \xi^n\|+\|\xi^n\| \right) \|\bar{\delta}_t\xi^n\| \leq C\| \nabla \xi^n\| \|\bar{\delta}_t\xi^n\|  \leq \frac{g_0}{8}\|\bar{\delta}_t\xi^n\|^2 + C\| \nabla \xi^n\|^2.\]
We further estimate the second term of the right-hand side of \eqref{eq7.1}. We apply \eqref{eq6.3} to get
\begin{equation}\label{eq7.11}
   \begin{split}
  &\left(G_0\left(u^n,\nabla u^n\right) \bar{\delta}_tu^n - G_0\left( w_h(t_n),\nabla w_h(t_n)\right) \bar{\delta}_tw_h(t_n), \bar{\delta}_t\xi^n \right)\\
  &\leq  \left| \left( G_0\left(u^n,\nabla u^n\right) [\bar{\delta}_tu^n -  \bar{\delta}_tw_h(t_n)], \bar{\delta}_t\xi^n\right) \right| \\
        & + \left| \left( [G_0\left(u^n,\nabla u^n\right) - G_0\left( w_h(t_n),\nabla w_h(t_n)\right)] \bar{\delta}_tw_h(t_n), \bar{\delta}_t\xi^n\right) \right| \\
        & \leq \mathcal{L}\left[ \mu_1(\|u^n\|^2-\|w_h(t_n)\|^2)+\mu_2(\|\nabla u^n\|^2-\|\nabla w_h(t_n)\|^2)\right]\\
        & \qquad  \times \|\bar{\delta}_tw_h(t_n)\|\|\bar{\delta}_t\xi^n\| + g_1 \left\|\bar{\delta}_t\eta^n\right\|\|\bar{\delta}_t\xi^n\| 
         \leq  C\left(\|\nabla \eta^n\| + \left\|\bar{\delta}_t\eta^n\right\| \right) \frac{\|\xi^n\|_A+\|\xi^{n-1}\|_A}{2}.
    \end{split}
   \end{equation}  
We invoke \eqref{eq7.2}--\eqref{eq7.11} in \eqref{eq7.1}, sum the resulting equation multiplied by $2\tau$ for $n$ from 1 to $m$ and use \eqref{eq6.2}--\eqref{eq6.7} to get
\begin{multline*}
       \|\xi^m\|^2_A  \leq  \|\xi^0\|^2_A
           + C \tau \sum\limits_{n=1}^{m}  \Big(\|\xi^n\|_A^2 +  \|\delta_t^2 \eta^n\| \frac{\|\xi^n\|_A+\|\xi^{n-1}\|_A}{2} 
        \\
         +  \left\|u'(t_n)-\bar{\delta}_tu^n\right\| \frac{\|\xi^n\|_A+\|\xi^{n-1}\|_A}{2} 
         +  \sum\limits_{j=1}^{4}\|R_j^n\| \frac{\|\xi^n\|_A+\|\xi^{n-1}\|_A}{2}  \\
        +   \left(\|\nabla \eta^n\| + \left\|\bar{\delta}_t\eta^n\right\| \right) \frac{\|\xi^n\|_A+\|\xi^{n-1}\|_A}{2}\Big).
   \end{multline*} 
Choose a suitable $\ell$ such that $\|\xi^\ell \|_A = \max\limits_{0\leq n \leq m}\left\|\xi^n \right\|_A$, we have for $j\le m,$ 
\begin{multline}\label{eq7.13}
       \|\xi^\ell\|_A  \leq  C \tau \sum\limits_{n=1}^\ell  \|\xi^n\|_A + \left\|\xi^0 \right\|_A
          + C \tau \sum\limits_{n=1}^\ell \|\delta_t^2 \eta^n\| 
        \\
         + C \tau \sum\limits_{n=1}^\ell\Big( \left\|u'(t_n)-\bar{\delta}_tu^n\right\|  +  \sum\limits_{j=1}^{4}\|R_j^n\|    +   \left(\|\nabla \eta^n\| +\left\|\bar{\delta}_t\eta^n\right\| \right)\Big).
   \end{multline}
Next, we will estimate the terms of the right-hand side of \eqref{eq7.13}. First, we get
\begin{align}
     & \tau \sum\limits_{n=1}^{m}  \left\|\delta_t^{2}\eta^n \right\|\leq
      C \tau \sum\limits_{n=1}^{m} \|\eta''_1(\kappa_n)\|\leq Ch \|u''\|_{L^{\infty}(H^1)}, \;\; \kappa_n\in [t_{n-1}, t_{n+1}], \label{eq7.14} \\
      &  \tau \sum\limits_{n=1}^{m} \left\|u'(t_n)-\bar{\delta}_tu^n\right\|  \leq \frac{\tau^2}{2}\int_0^{t_{m+1}}\|u'''(s)\|ds,\nonumber\\
      & \tau \sum\limits_{n=1}^{m} \left(\|\nabla \eta^n\| +\left\|\bar{\delta}_t\eta^n\right\| \right) \leq Ch\left(\|u\|_{L^{\infty}(H^2)}+\|u'\|_{L^{\infty}(H^1)}\right). \nonumber
\end{align}
Then we discuss the terms $R_j^n$ with $1\leq j\leq 4$. First, we bound the term on $R_1^n$ in \eqref{qq1}. We apply the Taylor expansion with integral remainder to arrive at
\begin{equation*}
  \begin{split}
    &  u''(t_n) - \delta_t^{2}u^n =  \frac{-1}{6\tau^2} \left[ \int_{t_n}^{t_{n+1}}(t_{n+1}-t)^3u^{(4)}(t)dt +   \int_{t_{n-1}}^{t_{n}}(t-t_{n-1})^3u^{(4)}(t)dt\right], \\
& n\geq 2, \quad  u''(t_1) - \delta_t^{2}u^1 =  \frac{-1}{2\tau^2} \left[ \int_{t_1}^{t_{2}}(t_{2}-t)^2 u'''(t)dt +   \int_{0}^{t_{1}}t^2u'''(t)dt\right],
  \end{split}
   \end{equation*}
which gives
\begin{equation*}
   \begin{split}
       \tau \sum\limits_{n=1}^{m} \|R_1^n\| \leq   \tau \int_0^{2\tau} \|u'''(t)\|dt + \tau^2 \int_{\tau}^{t_{m+1}} \|u^{(4)}(t)\|dt.
   \end{split}
   \end{equation*}
To estimate $R_2^n$ in \eqref{qq1}, we use the identity 
\begin{equation*}
   \begin{split}
        u'(t_n)-\Bar{\delta}_tu^n   &=  \frac{-1}{4\tau} \left[ \int_{t_n}^{t_{n+1}}(t_{n+1}-t)^2u'''(t)dt +   \int_{t_{n-1}}^{t_{n}}(t-t_{n-1})^2u'''(t)dt\right],
   \end{split}
   \end{equation*}
    and obtain 
\begin{equation*}
   \begin{split}
      \tau \sum\limits_{n=1}^{m}  \left\|R_2^n \right\|\leq
       \frac{g_1}{2} \tau^2 \int_0^{t_{m+1} }\|u'''(t)\|dt.
   \end{split}
   \end{equation*}
   However, to estimate  $R_3^n$ in \eqref{qq1}, we use
   \begin{equation*}
   \begin{split}
      \Delta u(t_n) - \Delta\widetilde{u}^n & =  \frac{-1}{2} \Bigg[ \int_{t_{n}}^{t_{n+1}} (t_{n+1}-t) \Delta u''(t)dt + \int_{t_{n-1}}^{t_{n}} (t-t_{n-1}) \Delta u''(t) dt \Bigg],
   \end{split}
   \end{equation*} 
   and get 
\begin{equation*}
   \begin{split}
       \tau \sum\limits_{n=1}^{m} \|R_3^n\|  \leq \tau^2 \int_0^{t_{m+1} }\|\Delta u''(t)\|dt.
   \end{split}
   \end{equation*}
To estimate  $R_4^n$ in \eqref{qq4},
we follow the procedure in \cite[52]{McLean} with $\hat{\mu}_j = \int_{t_j}^{t_{j+1}}K(s)ds$ to obtain 
\begin{equation*}
   \begin{split}
      \tau \sum\limits_{n=1}^{m}  \|R_4^n\|
       & \leq
       \hat{\mu}_0 \tau \int_0^{\tau} \|\Delta u''(t)\|dt + \tau^2 \sum\limits_{p=2}^{m}\left(\sum\limits_{n=p}^{m} \hat{\mu}_{n-p} \right) \int_{t_{p-1}}^{t_p} \|\Delta u'''(t)\|dt  \\
       & + \tau \sum\limits_{p=1}^{m} \Delta\left(u'(t_p)-\bar{\delta}_tu^p \right) \left( \sum\limits_{n=p}^{m} \widetilde{\kappa}_{np} \right).
   \end{split}
   \end{equation*}
Thus, since $\sum\limits_{n=p}^{m} \hat{\mu}_{n-p}  \leq \int_0^{t_m}K(t)dt \leq C$ and since $\sum\limits_{n=p}^{m} \widetilde{\kappa}_{np} \leq C\tau   \sum\limits_{n=p}^{m} \leq C$ (which is due to  \eqref{coeff}),   
\begin{equation*}
   \begin{split}
      \tau \sum\limits_{n=1}^{m}  \|R_4^n\|
       & \leq
       C \tau \int_0^{\tau} \|\Delta u''(t)\|dt + C\tau^2  \int_{\tau}^{t_{m+1}} \|\Delta u'''(t)\|dt. 
   \end{split}
   \end{equation*}
Finally, to estimate the first term of the right-hand side of \eqref{eq7.13}, i.e., $\|\xi^0\|_A$. we subtract \eqref{ModelE} from \eqref{eq2.16} and use \eqref{qq5} to get
\begin{equation*}
    \begin{split}
        & \delta_t \xi^1  = \delta_t \eta^1 + (u_{1h}-u_1) + \frac{\tau}{2}(u_{2h}-u_2) - \mathcal{R}^0, \\
        & \nabla  \xi^1 = \tau \left[ \delta_t \nabla\eta^1 + (\nabla u_{1h}-\nabla u_1) + \frac{\tau}{2}(\nabla u_{2h}-\nabla u_2) - \nabla\mathcal{R}^0 \right],
    \end{split}
\end{equation*}
which, together with \eqref{eq6.3}, leads to
\begin{equation}\label{eq7.22}
   \begin{split}
       \left\|\xi^0 \right\|_A & \leq Ch \left( \|u'\|_{L^{\infty}(H^1) } +  \|u_1\|_1 +\tau \|u_0\|_2 \right) +  C\tau \int_0^{\tau} \| u'''(t)\|dt.
   \end{split} 
   \end{equation}
We invoke \eqref{eq7.14}--\eqref{eq7.22} in \eqref{eq7.13} and use $\|\nabla \eta(t_m)\|\leq Ch\max\limits_{0\leq t \leq t_m}\|u(t)\|_2$ and the discrete Gr\"{o}nwall's lemma to obtain \eqref{eq2.20}. This completes the proof.
\end{proof}

\subsection{Discussion on temporal accuracy}
Since the solution of \eqref{eq1.1}-\eqref{eq1.3} may exhibit initial singularity and is in general smooth away from the initial time under certain regularity assumptions on the given data, we focus the attention near the initial time and thus consider  \eqref{eq1.1}-\eqref{eq1.3} with $q(t)= q(0):=q_0$ for simplicity. Hence, to consider the temporal solution regularity, one could apply the eigenpairs of $-\Delta$ to decompose the linear version of \eqref{eq1.1}-\eqref{eq1.3} as ordinary differential equations of the following form for some $\lambda>0$: for $t>0$, 
\begin{equation}\label{eq7.36}
        u'' + q_0 u'   +\lambda u -\lambda \int_0^{t}\beta(t-s) u(s)ds = f(t), ~~{\rm with}~~ u(0) = u_0~{\rm and}~u'(0) = u_1.
   \end{equation}
Since 
\[ u'(t)= u_1 +  \int_0^t u''(s) ds~{\rm and}~  u(t)= u_0 + tu_1 + \int_0^t (t-s)u''(s) ds, \quad t>0,\]
\begin{equation*}
   \begin{split}
        u''(t)  =   f(t) & - q_0 \left[u_1 +  \int_0^t u''(s) ds\right] - \lambda \left[u_0 + tu_1 + \int_0^t (t-s)u''(s) ds\right] \\
        & + \lambda \int_0^{t}\beta(t-s) \left[u_0 + su_1 + \int_0^s (s-z)u''(z) dz \right]ds, \quad t> 0.
   \end{split}
   \end{equation*}
Thus, for the case (i) in (\ref{eq1.5}), we have the asymptotic behaviour of $u''$ and $u'''$ as follows 
\begin{align*}
        u''(t)& \simeq f(t) - q_0 u_1 - \lambda u_0 + \frac{\lambda u_0}{\Gamma(1+\alpha)}t^{\alpha} + O(t^{\alpha}), \quad t\rightarrow 0^{+},\\
        u'''(t)& \simeq f'(t) - q_0\, u''(0) - \lambda u_1 + \frac{\lambda u_0}{\Gamma(\alpha)}t^{\alpha-1} + O(t^{\alpha-1}), \quad t\rightarrow 0^{+},
   \end{align*}
which implies $|u''(t)| \leq C$ and $|u'''(t)| \leq C(1 + t^{\alpha-1})$. Similarly we could get $|u^{(4)}(t)| \leq C(1 + t^{\alpha-2})$. For the case (ii) in (\ref{eq1.5}), where there is no singularity in the kernel, the solutions could be smooth under smooth data. Thus, it is reasonable to assume that the solutions to  \eqref{eq1.1}-\eqref{eq1.3} satisfy
\begin{equation}\label{regul}
   \begin{split}
       &t^{\alpha-1}\|\Delta u''(t)\| +\|\Delta u'''(t)\| \leq C(1+t^{\alpha-1}), \quad \|u^{(4)}(t)\| \leq C(1+t^{\alpha-2})\text{ for case } (i), \\
       & \|\Delta u''(t)\| +  \|\Delta u'''(t)\| + \|u^{(4)}(t)\| \leq C\text{ for case } (ii). 
   \end{split}
   \end{equation}
Then the last three terms on the right-hand side of \eqref{eq2.20} could be further bounded as 
\begin{equation*}
   \begin{split}
       & \int_0^{2\tau} \|u'''(t)\|dt + \tau \int_{\tau}^{t_{m+1}} \|u^{(4)}(t)\|dt  \leq C\tau^\alpha, \\
       & \int_0^{t_{m+1} }\Big(\|u'''(t)\| + \|\Delta u''(t)\|\Big)dt \leq C, \\
       &  \int_0^{\tau} \|\Delta u''(t)\|dt + C\tau  \int_{\tau}^{t_{m+1}} \|\Delta u'''(t)\|dt \leq C\tau,
   \end{split}
   \end{equation*}
which results in $O(\tau^{1+\alpha})$ temporal accuracy with $\alpha=1$ or $\alpha=\frac{1}{2}$.

\begin{remark}\label{rem2.1}
Motivated by the above analysis, Theorem \ref{theorem2.5} indicates the $O(h+\tau^{1+\alpha})$ accuracy of the $H^1$-norm error with $\alpha=1$ or $\alpha=\frac{1}{2}$. 
\end{remark}

\section{Numerical experiment} \label{sec5}

\subsection{One-dimensional case}
 Let $\Omega=(0,1)$, $u_0(x)=\sin(\pi x)$,  $u_1(x)=\sin(2\pi x)$ and $f(x,t)=t^{\alpha}e^{-\sigma t} \cos(\gamma t) \sin(\pi x)$. Let $G(z)=\sqrt{1+z}$ with $\mu_1=\mu_2=1$ in \eqref{eq1.2} and $h=\frac{1}{M}$ for some $M>0$. Since the exact solution is unknown, to illustrate numerically the achieved $H^1$-norm convergence rates in Theorem \ref{theorem2.5}, we define 
\begin{align*}
    {E_t(M,N)}& =\sqrt{h\sum\limits_{j=1}^{M-1}\left|V_j^{N+1}-V_j^{2N+1} \right|^2}, ~ {E_s(M,N)}=\sqrt{h\sum\limits_{j=1}^{M-1}\left|V_j^{N+1}-V_{2j}^{N+1} \right|^2},~V_j^n  = \frac{U_j^n-U_{j-1}^n}{h}\,.
\end{align*}
Here $U_j^n$ approximates $u(x_j,t_n)$ and hence,  $V_j^n$  approximates $\nabla u (x_j,t_n)$,
$E_t$ represents the difference between gradients of numerical solutions at time  $T$ computed under the time step sizes $\tau=T/N$ and $\tau=T/(2N)$, and $E_s$ could be interpreted similarly.
The convergence rates in time and space are accordingly defined as
\begin{align*}
    CR_t  =\log_{2}\left(\frac{E_t(M,N)}{E_t(M,2N)}\right), \quad CR_s=\log_{2}\left(\frac{E_s(M,N)}{E_s(2M,N)}\right).
\end{align*}

We  fix $T=1$ to evaluate the $H^1$ errors and convergence rates of the fully-discrete Galerkin scheme in Table \ref{tab1}, which demonstrate the $O(\tau^{1+\alpha} + h)$ accuracy as predicted in Theorem \ref{theorem2.5}. 


\begin{table}[H]
    \center \footnotesize
    \caption{ $H^1$ errors and time-space convergence rates under $T=1$.} \label{tab1}
    \vskip 2mm
    \begin{tabular}{cccccccccccc}
      \toprule
    & & & \multicolumn{2}{c}{$N=32$, $\gamma=3\sqrt{3}$, $\sigma=3$} & &\multicolumn{2}{c}{$M=32$, $\gamma=1.0$, $\sigma=2.0$}\\
   \cmidrule{4-5}  \cmidrule{7-8}
      $\alpha=0.5$ &$M$ & & {$E_s$} & {$CR_s$}  & $N$ & {$E_t$} & {$CR_t$} \\
      \midrule
       & $32$    & &  $2.4232 \times 10^{-2}$  &    *         & $128$    & $2.3145 \times 10^{-3}$  &    *    \\
       & $64$    & &  $1.2798 \times 10^{-2}$  &    0.92      & $256$    & $7.2638 \times 10^{-4}$  & 1.67   \\
       & $256$   & &  $6.5733 \times 10^{-3}$  &    0.96      & $512$    & $2.3469 \times 10^{-4}$  & 1.63    \\
       & $512$   & &  $3.3306 \times 10^{-3}$  &    0.98      & $1024$   & $7.7536 \times 10^{-5}$  & 1.60    \\
   \text{Predict}  &     & &                    &    1.00      &                                 &   & 1.50   \\
       \midrule 
        & & & \multicolumn{2}{c}{$N=32$, $\gamma=2$, $\sigma=2$} & &\multicolumn{2}{c}{$M=32$, $\gamma=0.5$, $\sigma=1.1$}\\
   \cmidrule{4-5}  \cmidrule{7-8}
      $\alpha=1.0$ &$M$ & & {$E_s$} & {$CR_s$}  & $N$ & {$E_t$} & {$CR_t$} \\
      \midrule
       & $16$   & &  $3.6308 \times 10^{-2}$  &    *         & $16$    & $1.2054 \times 10^{-1}$  &   *    \\
       & $32$   & &  $1.8470 \times 10^{-2}$  &    0.98      & $32$    & $2.8221 \times 10^{-2}$  & 2.09   \\
       & $64$   & &  $9.5854 \times 10^{-3}$  &    0.95      & $64$    & $6.8974 \times 10^{-3}$  & 2.03    \\
       & $128$  & &  $4.9102 \times 10^{-3}$  &    0.97      & $128$   & $1.7148 \times 10^{-3}$  & 2.01    \\
   \text{Predict}  &     & &                   &    1.00      &                                &   & 2.00   \\
      \bottomrule
    \end{tabular}
\end{table}

Next, we test the possible energy dissipation of the proposed model. It is shown in \cite[Page 498]{Cavalcanti1} that the Euler-Bernoulli viscoelastic model, which coincides with the form of model (\ref{eq1.1}) with $-\Delta$ replaced by $\Delta^2$, admits the energy dissipation with the corresponding energy 
\begin{equation}\label{mh1}
\hat E(t):=\frac{1}{2} \|u'(t)\|^2 + \frac{1}{2} \|\Delta u(t)\|^2
\end{equation}
Motivated by this definition, we define the following energy for problem \eqref{eq1.1} with $f=0$ by 
\begin{equation}\label{energy1}
  \begin{split}
     \widetilde{E}(t) = \frac{1}{2} \|u'(t)\|^2 + \frac{1}{2} \|\nabla u(t)\|^2, \quad t\geq 0,
  \end{split}
\end{equation}
and we numerically evaluate this energy as follows 
\begin{equation*}
  \begin{split}
     \widetilde{E}^n = \frac{1}{2} \|\bar{\delta}_tU^n\|^2 + \frac{1}{2} \|\nabla U^n\|^2, \quad n\geq 1, \quad \widetilde{E}^0 = \frac{1}{2} \|u_1\|^2 + \frac{1}{2} \|\nabla u_0\|^2.
  \end{split}
\end{equation*}
Numerical results are presented in Figure \ref{fig1}, which indicate the dissipative property of the energy for $\alpha=0.5$ or $1$.

\begin{figure}[H]
\centering
\includegraphics[width=3.8in]{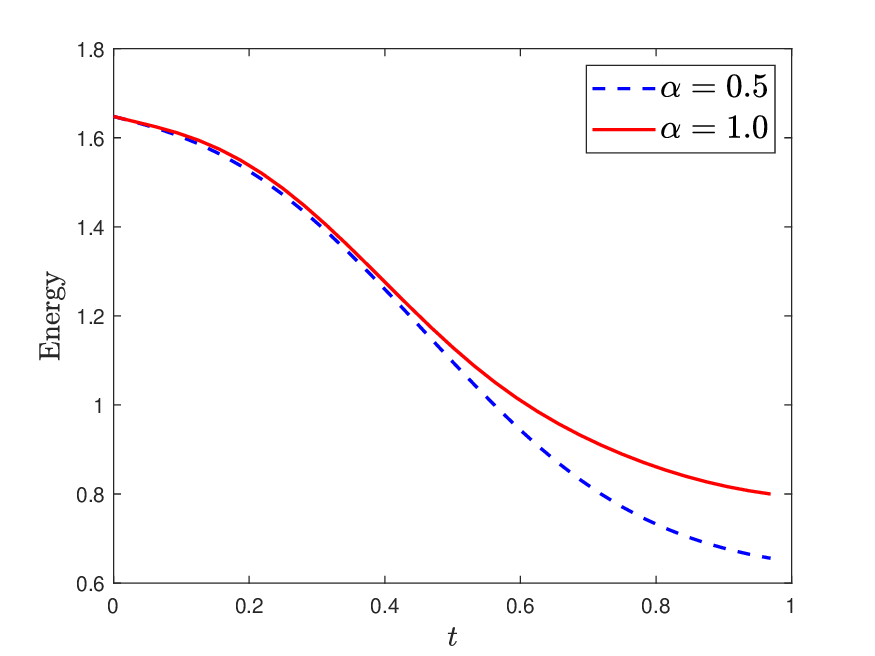}
\caption{Plots of the energy under $f=0$, $\sigma=3$, $\gamma=3\sqrt{3}$, and $N=M=32$.}
\label{fig1}
\end{figure}

\subsection{Two-dimensional case} 
Here, let $\Omega=(0,1)\times (0,1)$, $u_0(x,y)=\sin(\pi x)\sin(\pi y)$,  $u_1(x,y)=\sin(2\pi x)\sin(2\pi y)$, and $f(x,y,t)=0$. Let $G(z)=\sqrt{1+z}$ with $\mu_1=\mu_2=1$ in \eqref{eq1.2} and $h_x=h_y=h=\frac{1}{M}$ for some $M>0$. In two-dimensional case, to illustrate numerically the achieved $H^1$-norm convergence rates in Theorem \ref{theorem2.5}, we denote
\begin{align*}
    {E_t(M,N)}& =h\,\sqrt{\sum\limits_{i=1}^{M-1}\sum\limits_{j=1}^{M-1}\left|W_{i,j}^{N+1}-W_{i,j}^{2(N+1)} \right|^2},  \quad {E_s(M,N)}=h\,\sqrt{\sum\limits_{i=1}^{M-1}\sum\limits_{j=1}^{M-1}\left|W_{i,j}^{N+1}-W_{2i,2j}^{N+1} \right|^2},   \\
    W_{i,j}^n & = \sqrt{\left(\frac{U_{i,j}^n-U_{i-1,j}^n}{h}\right)^2+\left(\frac{U_{i,j}^n-U_{i,j-1}^n}{h}\right)^2}.
\end{align*}

In Table \ref{tab3-1}, we test the $H^1$ errors and convergence orders of the fully discrete Galerkin scheme for the two-dimensional case. We observe that the numerical results are consistent with our theoretical analysis (see Remark \ref{rem2.1}).

\begin{table}
    \center \footnotesize
    \caption{ $H^1$ errors and time-space convergence rates under $T=1/2$.} 
    \label{tab3-1}
    \vskip 2mm
    \begin{tabular}{cccccccccccc}
      \toprule
    & & & \multicolumn{2}{c}{$N=16$, $\gamma=3\sqrt{3}$, $\sigma=3$} & &\multicolumn{2}{c}{$M=64$, $\gamma=0.5$, $\sigma=1.5$}\\
   \cmidrule{4-5}  \cmidrule{7-8}
      $\alpha=0.5$ &$M$ & & {$E_s$} & {$CR_s$}  & $N$ & {$E_t$} & {$CR_t$} \\
      \midrule
       & $64$    & &  $6.9769 \times 10^{-3}$  &    *         & $64$    & $4.2265 \times 10^{-3}$  &    *    \\
       & $128$   & &  $3.3304 \times 10^{-3}$  &    1.07      & $128$   & $1.3363 \times 10^{-3}$  & 1.66   \\
       & $256$   & &  $1.6259 \times 10^{-3}$  &    1.03      & $256$   & $4.3567 \times 10^{-4}$  & 1.62    \\
       & $512$   & &  $8.0321 \times 10^{-4}$  &    1.02      & $512$   & $1.4597 \times 10^{-4}$  & 1.58    \\
   \text{Predict}  &     & &                   &    1.00      &                                &   & 1.50   \\
       \midrule 
        & & & \multicolumn{2}{c}{$N=16$, $\gamma=2$, $\sigma=2$} & &\multicolumn{2}{c}{$M=64$, $\gamma=0.5$, $\sigma=1.1$}\\
   \cmidrule{4-5}  \cmidrule{7-8}
      $\alpha=1.0$ &$M$ & & {$E_s$} & {$CR_s$}  & $N$ & {$E_t$} & {$CR_t$} \\
      \midrule
       & $64$   & &  $8.9036 \times 10^{-3}$  &    *       & $32$   & $9.9870 \times 10^{-3}$  &   *    \\
       & $128$  & &  $4.3094 \times 10^{-3}$  &    1.05    & $64$   & $2.4006 \times 10^{-3}$  & 2.06   \\
       & $256$  & &  $2.1189 \times 10^{-3}$  &    1.02    & $128$  & $5.1903 \times 10^{-4}$  & 2.21    \\
       & $512$  & &  $1.0505 \times 10^{-3}$  &    1.01    & $256$  & $1.2184 \times 10^{-4}$  & 2.09    \\
   \text{Predict}  &    & &                   &    1.00    &                               &   & 2.00   \\
      \bottomrule
    \end{tabular}
\end{table}

\section{Concluding remarks} \label{sec6}

This work performs numerical analysis for a hyperbolic integrodifferential equation. The developed techniques in this work overcome the difficulties caused by the non-positive variable-sign kernel and nonlinear-nonlocal damping, which could also be extended to investigate other related problems such as the Euler-Bernoulli viscoelastic problem in \cite{Cavalcanti1} that coincides with the form of model (\ref{eq1.1}) with $-\Delta$ replaced by $\Delta^2$, or (\ref{eq1.1}) with a different damping coefficient \cite{Xu3}
    \begin{equation*}
        q(t) = G\left( \int_0^t   \int_{\Omega}\mu_1| u(x,s)|^2+\mu_2|\nabla u(x,s)|^2dxds  \right), \;\; \mu_1, \mu_2 \geq 0, \;\; \mu_1^2+ \mu_2^2 \neq 0.
    \end{equation*}
    In particular, it is shown in \cite[Theorem 2.1]{Cavalcanti1} that if the kernel function $\beta(t)$ satisfies the assumption ($\mathbf{A1}$) and
   \begin{align}
       &\qquad\qquad \beta(0)>0, \quad 1-\int_{0}^{\infty}\beta(s)ds = \ell >0, \label{ass1} \\
       & -c_1\beta(t) \leq \beta'(t) \leq -c_2 \beta(t), \quad 0\leq \beta''(t) \leq c_3 \beta(t) \quad \forall t\geq 0, \nonumber 
    \end{align}
    where $c_1,c_2,c_3$ are positive constants, the energy (\ref{mh1}) decays exponentially in time. However, since the kernel $\beta(t)$ in this work has a variable sign and initial singularity, it does not satisfy the conditions in \eqref{ass1} such that it is difficult to follow the derivations in \cite{Cavalcanti1} to derive the energy decay. Nevertheless, the numerical experiment suggests the energy decay as shown in Figure \ref{fig1}, which motivates us to perform a further investigation for this issue.

Finally, in the current work we only prove the error estimate in the $H^1$ norm. Due to the existence of the gradient term in $q(t)$, which may limit the improvement of the convergence order in deriving the $L^2$ error, it is not straightforward to obtain the optimal $L^2$ error estimate in the current circumstance. We will investigate this interesting question in the near future.



\end{document}